\documentclass[12pt]{article}
\usepackage{amssymb,amsthm,amsmath}
\usepackage{latexsym}
\usepackage{graphicx}

\newtheorem{theorem}{Theorem}[section]

\newtheorem{thm}{Theorem}
\newtheorem{claim}[theorem]{Claim}
\newtheorem{conj}[theorem]{Conjecture}
\newtheorem{cor}[theorem]{Corollary}
\newtheorem{obs}[theorem]{Observation}
\newtheorem{problem}[theorem]{Problem}
\newtheorem{defi}[theorem]{Definition}

\title{Path-search in the pyramid and in other graphs}
\author{D\'aniel Gerbner \and Bal\'azs Keszegh}

\begin{document}

\maketitle

\begin{abstract}
We are given an acyclic directed graph with one source, and a subset of its edges which
contains exactly one outgoing edge for every non-sink vertex. These edges
determine a unique path from the source to a sink. We can think of it as
a switch in every vertex, which determines which way the water
arriving to that vertex flows further.

We are interested in determining either the sink the flow arrives,
or the whole path, with as few questions as possible. The
questions we can ask correspond to the vertices of the graph, and
the answer describes the switch, i.e. tells which outgoing edge is
in our given subset.

Originally the problem was proposed by Soren Riis (who posed the question for pyramid graphs) in the following more general form. We are given a natural
number $k$, and $k$ questions can be asked in a round. The goal is
to minimize the number of rounds. We completely solve this problem for complete $t$-ary trees. Also, for pyramid
graphs we present some non-trivial partial results.
\end{abstract}

\section{Introduction}

In this paper we consider the following problem. We are given
a directed acyclic (multi)graph $G$ with one source. We are given
a set $E^* \subset E$ such that for every non-sink vertex $v\in V$
there is exactly one outgoing edge in $E^*$. It determines a path
from the source to a sink. We can think of it as a switch in every
vertex, which determines which way the water (coming from the source) arriving to that
vertex flows further. We imagine such graphs such that the source is on the top and the sinks are on the bottom. In the rest of the paper by a graph we always mean such a directed graph except when noted.

We are interested in determining either the whole path $P=P(E^*)$ or only the sink the flow arrives, i.e. the last vertex of $P(E^*)$. The questions we can ask correspond
to the vertices of the graph, and the answer describes the switch,
i.e. tells which outgoing edge is in $E^*$.

\begin{defi}
Let $pa(G)=pa_1(G)$ (resp. $si(G)=si_1(G)$) be the minimal number of questions we need
to determine the path (resp. the sink).
\end{defi}

Clearly $si(G) \le pa(G)$, but these can be very far from each
other, for example if there is only one sink, then $si(G)=1$ but
$pa(G)$ can be arbitrarily high.

Soren Riis proposed the following problem. What is the minimal
number of rounds we need if in one round we can ask $k$ questions
and our aim is to determine the path or the sink in the pyramid
graph (for the definition see Section \ref{pypa})? This
motivates the study of the following more general question.

\begin{defi}
Fixing a $k$, $pa_k(G)$ is the minimal number of rounds we need to
determine the path, if in each round we can ask $k$ questions.
Similarly, $si_k(G)$ is the minimal number of rounds we need to
determine the sink, if in each round we can ask $k$ questions.
\end{defi}

As an additional motivation, let us consider a game and two
possible strategies given as black boxes, i.e. we can evaluate in
every possible state what the strategies do. To represent the game
with a directed acyclic graph, it is enough represent every valid
(turn number, state) pair with a different vertex, and the valid
steps are represented by the directed edges. For every finite
game, this is a finite acyclic directed graph. Now we can simulate
a match between two players using the respective strategies. To
find out which strategy is the winner of a match, one can go step
by step and ask what happens in every actual situation. Note that
it does not matter if the black box strategies use randomness or
not. This process is a very important step of Monte-Carlo type
algorithms, when one needs to quickly simulate matches between
strategies \cite{brugmann, montecarlo, ksz}.

However, suppose we are given multiple processors. A solution to
the original sink-search problem for the above defined graph gives
an optimal parallel algorithm to determine the final state of the
match. Clearly, such an algorithm determines the winner of the
match as well. Summarizing, investigating our problem can help
find faster algorithms for simulating matches, if parallel
computations are allowed. (We note that we do not think it could
help a lot. We were able to find an example where the optimal
algorithm is much faster than the trivial one, but it was a very
special graph. Also in case of games usually the graph is very
large, which makes it hard to find a good non-trivial algorithm
for our problem.) For further motivations related to random walks
see the beginning of Section \ref{pypa}.

In the next section we examine $si(G)$ and $pa(G)$ in directed
acyclic graphs. In Section \ref{tree} we consider $si_k(G)$ and
$pa_k(G)$ on trees and completely solve the problem on $d$-ary
trees. In Section \ref{pypa} we consider the problem on the
pyramid graph. In the last section we conclude our paper with some
additional remarks and open problems.

\section{Search in directed acyclic graphs}\label{dag}

We will show a process how to transform the graph $G$ into a graph
$G'$ such that $si(G)=si(G')$ and $si(G')$ can be easily
determined. We also show how to transform $G$ into a graph $G''$
such that $pa(G)=pa(G'')$ and $pa(G'')$ can be easily determined.

But at first let us consider a trivial algorithm which finds both
the sink and the path: at first the source is asked. Then the
answer tells us which vertex is the next on the path. That one is
asked in the next turn, and so on. This gives us the following
simple observation.

\begin{obs}\label{trivi} For every (multi)graph $G$, $si(G)\le pa(G)\le l(G)$,
where $l(G)$ is the length (the number of edges) of the longest
path.

\end{obs}

We start with examining $si(G)$. We can simply forget about
multiple edges then, hence we suppose $G$ is simple. Obviously
it's useless to ask a vertex with outdegree 1. We will define a
graph $G'$ with no such vertices such that if we find the sink in
$G'$ it gives us the sink in $G$ (in fact it is the same vertex).
We introduce the following {\em merging} operation: for a set of
vertices $M$, we get $G(M)$ from $G$ by deleting the vertices $M$
and introducing a new vertex $m$, if there was an edge between a
vertex $v$ in $G\setminus M$ and a vertex $w$ in $M$ then we put
an edge in the new graph between $v$ and $m$ with the same
orientation as in $G$. If multiple edges come into existence, we
consider them as one edge. Edges between vertices of $G\setminus
M$ stay untouched. Acyclicity could be ruined by such an operation
but anytime we do such an operation, it will be easy to see that
acyclicity remains true.

If a vertex $x$ of $G$ has exactly one out-neighbor $y$, then
$G({\{x,y\}})$ remains acyclic and $si(G({\{x,y\}}))=si(G)$. By
merging this way vertices with outdegree $1$ with their
outneighbors as long as it is possible, we get a graph $G'$ with
vertices all having outdegree minimum $2$ (except the sinks) and
for which $si(G')=si(G)$. Although it is not relevant for us, it
is easy to see that $G'$ does not depend on the order in which we
process the vertices. The merging procedure defines a map from the
vertices of $G$ to the vertices of $G'$, hopefully without causing
confusion, we will refer to the image of some vertex $x$ of $G$
also as $x$. By the procedure, every vertex of $G'$ has at least 2
out-neighbors or it is a sink.

\begin{obs}\label{obs2} If a vertex $y$ can be reached from a vertex $x$ in
$G$, then $y$ can be reached from $x$ in $G'$ as well.
\end{obs}









The main result of this section is that after merging every vertex
with outdegree 1 to get $G'$, $si(G')$ can be determined easily:
\begin{thm}\label{pasi} Suppose there is no vertex with outdegree 1 in a simple graph $G$.
Then $si(G)=l(G)$, where $l(G)$ is the length (the number of edges) of the longest
path.

\end{thm}

\begin{proof}
We need some preliminary observations. Let us examine what happens when a question is answered.
Clearly $xy \in E^*$ is equivalent to the following: all the other
edges starting from $x$ are not in $E^*$. But then we can simply
delete these edges to get a new graph and then the outdegree of $x$ becomes $1$, hence we can
merge $x$ with $y$ to get another graph which we denote by $G_{xy}$, ie. this answer reduced the problem to finding the sink in $G_{xy}$, which is the graph we get by deleting from $G$ all the edges going out from $x$ except $xy$ and then merging $x$ and $y$. Thus, if the first question is $x$ and the
answer is $xy$, then $si(G_{xy})$ additional questions are needed
to find the sink.

\begin{defi} A directed path is called full if it ends in a sink.
\end{defi}

Let us consider $G_{xy}'$. By definition $x$ and $y$ are merged.
Then if a vertex $z$ had only two outneighbors, $x$ and $y$, it
gets merged with them. Then if all the outneighbors of another
vertex are among $x$, $y$ and $z$ it also gets merged with them,
and so on. Let $M$ be the set of vertices $M$ of $G$ for wich
every full path starting at some vertex $u$ in $M$ contains $x$ or
$y$. By the above argument, for this $M$, $G_{xy}'=G_{xy}(M)$.

\begin{obs}\label{obs3}  If in $G$ there is a full path starting at some vertex $z$ and avoiding $x$ and $y$, then when merging $x$ and $y$, $z$ cannot be among the merged vertices $M$.
\end{obs}

Note that if $G$ does not have vertices with outdegree 1, then if a vertex $m$ is among the merged vertices $M$, then there are paths to both $x$ and $y$ from $u$, hence by acyclicity of $G$ there are no paths from $x$ to $u$.
\begin{obs}\label{obs3b}
If there is a path from $x$ to $z$ in $G$, then when merging $x$ and $y$, $z$ cannot be among the merged vertices $M$.
\end{obs}

Now we can start to prove the theorem.

By Observation \ref{trivi} $si(G)\le pa(G)\le l(G)$, hence it
is enough to prove $l(G)\le si(G)$. We prove it by induction on
$k=l(G)$. The case $k=1$ can be easily seen.

Now we describe the strategy of the adversary. Let $P$ be a path
of length $k$, consisting of the edges $x_1x_2, x_2x_3, \dots,
x_kx_{k+1}$.

{\bf Case 1.} Suppose that the first vertex asked is $x_i$, and
there is an edge $x_{i-1}x_{i+1}$ in the graph. Then the
adversary's answer should be an edge $x_iy$, where $y \neq
x_{i+1}$. It means we continue with the graph $G_{x_iy}$. It might
contain vertices with outdegree 1, hence we need to determine
$G_{x_iy}'$.


By Observation \ref{obs3} the vertices $x_j$
($j \neq i$) do not get merged in $G_{x_iy}'$, thus there is a
path of length $k-1$ in $G_{x_iy}'$, containing $x_1, \dots,
x_{i-1},$ $x_{i+1}, \dots, x_{k+1}$ in this order. Then by
induction at least $k-1$ additional questions are needed in
$G_{x_iy}'$.

{\bf Case 2.} The first vertex asked is $x_i$ but there is no edge
$x_{i-1}x_{i+1}$ in the graph $G$. Then the adversary's answer
should be the edge $x_ix_{i+1}$. No vertex $x_j$ with $j\ge i$
gets merged by Observation \ref{obs3b}. Also, from $x_{i-1}$ there
is an edge going out which is not the edge $x_{i-1}x_i$. Going
along this edge we can find a path $Q$. First we claim that for
any choice of such a $Q$, it does not contain $x_{i}$. Indeed,
otherwise the original path $P$ minus the edge $x_{i-1}x_{i}$ plus
this path $Q$ would give a longer path in $G$, a contradiction.
Further, $Q$ can be chosen such that it avoids $x_{i+1}$ as well.
Indeed, every outdeegree is at least $2$, so we always have at
least one choice different from $x_{i+1}$ to continue. This way,
for every vertex $x_j$ with $j<i$ there is a full path starting at
$x_j$ and avoiding both $x_i$ and $x_{i+1}$, thus by Observation
\ref{obs3} none of these vertices gets merged in
$G_{x_ix_{i+1}}'$. Hence the path $x_1,\dots, x_{i-1},$ $ x_{i+1},
\dots, x_{k+1}$ is a $k-1$ long path in $G_{x_ix_{i+1}}'$, the
induction can be applied.

{\bf Case 3.} The first vertex asked, $x$ is not in the path $P$.
If it has an outneighbor $y$ not on the path, that should be the
adversary's answer. Then clearly the full path $P$ is in $G_{xy}'$
and by Observation \ref{obs3} none of its vertices are merged.
Hence by induction $k$ more questions are needed.

Thus we can suppose that all the outneighbors are on $P$. Let
$x_i$ be the first and $x_j$ be the last among them. Then the
answer of the adversary is $x_j$. Again, by Observation
\ref{obs3b}, no $x_l$ with $l>j$ gets merged. Similarly as in Case
2, any full path starting at $x_i$ must avoid $x$ and also we can
choose such a path which avoids $x_j$ as well. This path $Q$ shows
by Observation \ref{obs3} that $x_i$ won't be merged. By adding to
$Q$ the appropriate part of the path $P$, for any $x_l$, $l<i$ we
can build a path avoiding both $x$ and $x_j$, thus showing that
these vertices won't be merged neither. Finally, for a vertex
$x_l$ with $i<l<j$, $x_l$ can be reached by a directed path from
$x$ (starting with the edge $xx_i$ and then going along $P$), thus
by Observation \ref{obs3b} such a vertex cannot be merged neither.

As none of the vertices of the path are merged in $G_{xx_j}'$, by induction $k$ more questions are needed.
\end{proof}

\begin{cor} For any (multi)graph $G$, $si(G)=si(G')=l(G')$.
\end{cor}

Theorem \ref{pasi} implies also that for a simple graph $G$, $pa(G)=l(G)$, if all non-sink vertices of $G$ have outdegree at least 2. We will prove that this holds even for multi-graphs.

We examine the path-search problem for multigraphs, hence we can claim again: it's
useless to ask a vertex with outdegree 1.  We introduce the
following modified {\em merging} operation for multigraphs: for a set of vertices $M$, we get $G[M]$ from $G$ by deleting the vertices $M$ and introducing a new vertex $m$, if there was an edge between a vertex $v$ in $G\setminus M$ and a vertex $w$ in $M$ then we put an edge in the new graph between $v$ and $m$ with the same orientation as in $G$. Now, differently from the previous merging operation, if parallel edges come into existence, we consider them as parallel edges. Edges between vertices of $G\setminus M$ stay untouched.

Merging a vertex $x$ with outdegree $1$ with its outneighbor $y$
to get $G[{\{x,y\}}]$, the graph remains acyclic and
$pa(G[{\{x,y\}}])=pa(G)$. By merging this way vertices with
outdegree one with their outneighbors as long as it is possible,
we get a graph $G''$ with vertices all having outdegree minimum
$2$ (except the sinks) and for which $pa(G'')=pa(G)$.

\begin{thm} \label{pasi2} Suppose there is no vertex with outdegree 1 in a (multi)graph $G$. Then $pa(G)=l(G)$.
\end{thm}
\begin{proof}
We get $pa(G)\le l(G)$ from Observation \ref{trivi} so we only
have to deal with the other direction. We examine what happens
when a question is answered. Having $xy \in E^*$ is still
equivalent to the following: all the other edges starting from $x$
are not in $E^*$. But then we can simply delete these edges to get
a new graph and then the outdegree of $x$ becomes $1$, hence we
can use our new merging operation $G[\{x,y\}]$ to get another
graph which we denote by $G^{xy}$, ie. this answer reduced the
problem to finding the sink in $G^{xy}$, which is the graph we get
by deleting from $G$ all the edges going out from $x$ except $xy$
and then merging $x$ and $y$. Thus, if the first question is $x$
and the answer is $xy$, then $pa(G^{xy})$ additional questions are
needed to find the sink. Note that $(G^{xy})''=G^{xy}$. Now, if the asked vertex is from the fixed maximal path $P$, then the adversary answers the edge of the path, otherwise it answers arbitrarily. Thus, after any question and answer only at most one edge of $P$ gets merged, thus the
length of the maximal path reduces by at most one. From this by
induction it easily follows that the size of the maximal path is a
lower bound to the number of questions needed.
\end{proof}

\begin{cor} \label{pacor}
For any (multi)graph $G$, $pa(G)=pa(G'')=l(G'')$.
\end{cor}

Another interesting question is the case when the graph contains a
directed cycle. In this case it can happen that following the flow we get stucked in a directed
cycle and never go to a sink.

We begin with the path-search problem. Let us call a set of edges
a \emph{generalized path} if it contains a directed path and
possibly an additional edge which goes from the last vertex to a
vertex already on the path. Let $l'(G)$ denote the length (number
of edges) of the longest generalized path starting at the start
vertex.

One can easily see that if a multigraph $G$ does not contain any
vertices with outdegree $1$, then $pa(G)=l'(G)$. Indeed, we can
copy the proof of Theorem \ref{pasi2}. Consider a graph $G$ and a
generalized path $P$ of length $l'(G)$. No matter what the
question and the answer are, at most two vertices are merged,
hence at least $l'(G)-1$ vertices of $P$ remain. Again, the adversary answers in a way that for the
vertices of the fixed maximal generalized path $P$ the answer is always an edge of $P$. Then the
remaining vertices of $P$ form a generalized path of length at least $l'(G)-1$, and the
induction can be applied. Similarly, Corollary \ref{pacor} remains
true as well.

On the other hand, in case of $si(G)$ we encounter problems as our
basic operation for handling an answer cannot be applied, as we
might lose some information at every merging. For example let us
suppose $z$ has only two outneighbors, $x$ and $y$. Then $z$ gets
merged with them in $G_{xy}$. However, if $z$, $x$ and $y$ are
part of the cycle where the flow ends in this order, then after merging $x$ and $y$ there is no
way to differentiate this cycle from the other cycle where $y$ comes
immediately after $z$.


\section{Search in trees}\label{tree}

We consider trees as rooted directed trees where the edges are directed away from the root.
For complete $d$-ary trees on $n+1$ levels (denoted by $T_d(n)$) the obvious algorithm that asks as many complete levels as possible from the beginning, is the best possible:

\begin{thm}\label{treethm}
$pa_k(T_d(n))=si_k(T_d(n))=\lceil n/ \log'_d k \rceil$, where $\log'_d$ is defined as the biggest integer $i$ for which $1+d+d^2+\ldots+d^{i-1}\le k$ holds.
\end{thm}

\begin{proof}
Clearly, if we ask in each round the first as many full levels as possible, in each round $P_{max}$ gets longer by $\log'_d$ in worst case, thus in $\lceil n/ \log'_d k \rceil$ rounds we can easily find the path.
We give a strategy to the adversary, so that his answers force that from $k$ queries of any round, only at most $\log'_d$ will be on the final path $P$. This way, as the final path has length $n$, there were indeed at least $\lceil n/ \log'_d k \rceil$ queries altogether.

Now fix a round $i$, define $S$ as the set of vertices asked in this round. Now take a vertex $v\in S$ asked in this round. This vertex has $d$ children, taking these as roots, they define $d$ maximal subtrees of $T$, the $j$th such tree containing $r_j$ further vertices of $S$. For $v$, we answer the edge that goes to a child that has the minimal $r_j$ value. This defines an answer to every vertex in round $i$.

We claim that this way at most $\log'_d$ vertices from round $i$ will be on the final path $P$. Suppose that there are $x$ vertices from round $i$ that are on $P$. In reverse order we denote them by $q_0,q_1,\ldots q_{x-1}$, where $q_{x-1}$ is the one closest to the root. Define $Q_j$ as the maximal subtree of $T$ having $q_j$ as its root. We prove by induction on $j$ that $|Q_j\cap S|\ge 1+d+d^2+\ldots+d^j$. For $j=0$ this is obvious. For a general $j$ we define the trees $R_1,\ldots R_d$, where $R_l$ is the maximal tree having $q_j$'s $l$'th child as its root. Wlog. assume that the edge of $P$ at $q_j$ goes to the root of $R_1$. Now by the way the adversary answers, we know that $r_l=R_l\cap S$ is minimal for $l=1$ and as $R_1$ contains the vertex $q_{j-1}$ and everything below it, by induction $|R_1\cap S|\ge |Q_{j-1} \cap S|\ge 1+d+d^2+\ldots+d^{j-1}$. Summing this up, $|Q_i\cap S|=1+|R_1\cap S|+|R_2\cap S|+\ldots |R_d\cap S|\ge 1+d(1+d+d^2+\ldots+d^{j-1})=1+d+d^2+\ldots+d^{j}$ as claimed. Finishing the proof, $k=|S|\ge |Q_{x-1}\cap S|\ge 1+d+d^2+\ldots+d^{x-1}$ implies that $x\le \log'_d k$.
\end{proof}

As $1+2+2^2+\ldots+2^{i-1}=2^{i}-1$, for $d=2$ the formula in Theorem \ref{treethm} can be simplified:
\begin{cor}
For a complete binary tree on $n$ levels, $pa_k(T_2(n))=si_k(T_2(n))=\lceil n/ \lfloor\log_2 (k+1)\rfloor \rceil$
\end{cor}

We remark that for general trees, the obvious algoritm is not always the best possible. For example take the tree that starts from the root with a path $p_1p_2\ldots p_n$ of length $n$, each $p_i$ having one further child except $p_n$, where there is a complete $n$ level binary tree with root $p_n$. Now for $k=2$ the obvious algorithm asking always close to the root, increases $P_{max}$ by $2$ in each round until reaching $p_n$, then while processing the binary tree, in each round it can only increase $P_{max}$ by $1$, thus it finishes approximately in $3n/2$ rounds. On the other hand, if in each round we ask one vertex from the path and one vertex from the remainder of the binary tree, in each round the path shortens by one and the binary tree has one less levels. Thus in $n$ rounds we can finish with both parts.

\section{Search in pyramid paths}\label{pypa}

The {\em pyramid graph} $Py(n)$ is a directed graph defined
in the following way. $Py(n)$ has $N=n(n+1)/2$ vertices on $n+1$
levels, for $1\le i\le n+1$ the $i$th level having $i$ vertices
$v_{i,1},v_{i,2}\ldots v_{i,i}$, and from every vertex $v_{i,j}$
where $1\le i\le n$ and $1\le j\le i$, there is a {\em left
outgoing edge} going to $v_{i+1,j}$ (its {\em left child}) and a
{\em right outgoing edge} going to $v_{i+1,j+1}$ (its {\em right
child}). $Py(n)$ has one {\em root} on the top, $v_{1,1}$ and $n+1$
{\em sinks} on the bottom, the vertices on the $(n+1)$th level.

Let us suppose we are given a one dimensional random walk and we
want to find either the endpoint, or the whole walk. It is not
obvious what search model makes sense here. If we can ask which
way it goes in each step, then clearly we need to ask every step,
and that is enough, the order of the questions do not matter. But
suppose that a question is the following: which direction does the
walk go at the $i$th step if it is in the $j$th position?

Clearly it is equivalent to our model on the pyramid graph.


\begin{obs}
For arbitrary $k$ and $n$ we have
\begin{itemize}
\item[(a)] $pa_1(Py(n))\le pa_k(Py(n))$,
\item[(b)] $pa_{k+1}(Py(n))\le pa_k(Py(n))$,
\item[(c)] $si_{k+1}(Py(n))\le si_k(Py(n))$.
\end{itemize}
\end{obs}

It is trivial that if $k<N$ then in both cases we need at least
$2$ rounds. I.e. the non-adaptive version (having 1 round) of both problems
needs $N$ queries. Indeed, for $k<N$ there is a vertex that we did
not ask and so the adversary can answer in a way that the path
leads to this vertex, everything below this vertex is in a left state and so
the state of this non-asked vertex would determine the path and
the sink as well.

The fully adaptive version of the problem is again pretty simple.
\begin{claim}
$si_{1}(Py(n))=pa_{1}(Py(n))=n$.
\end{claim}
It follows from Theorem \ref{pasi} but we also give a simple proof specific to pyramid paths.
\begin{proof}
In $n$ rounds it is easy to determine the path and its sink.
First we ask the root and then according to the answer, its left
or right outgoing neighbor. We continue this way, in the $i$th
round determining the $i$th edge of the path, finally asking a
vertex on the $n$th level, thus determining the whole path.

Suppose now that we asked less than $n$ questions, then there is
a level $i\le n$, from which we did not ask any vertex. The
adversary answers always left and also at the end he tells us that
any vertex not on level $i$ is in left state. Thus we know which
vertex of level $i$ is on the path, and the state of this vertex
(which we don't know) would determine the path and also its sink.
\end{proof}

\begin{conj} \label{mainconj}
$si_{s_l}(Py(n))=pa_{s_l}(Py(n))=\lceil n/l\rceil$ if
$s_l=1+2+\ldots+l$ for some $l$.
\end{conj}

The upper bound holds by a simple algorithm:

\begin{claim} \label{upperbound}
$si_{s_l}(Py(n))\le pa_{s_l}(Py(n))]\le\lceil n/l\rceil$ if
$s_l=1+2+\ldots+l$ for some $l$.
\end{claim}

\begin{proof}
The algorithm is recursive. In the first round we ask the $s_l$
vertices that are on the first $l$ level, thus we will know the
first $l$ edges of the path and also we know the vertex $u$ on the
$(l+1)$st level that is on the path. Now take the new pyramid
graph with $n$ levels with root $u$, by recursion we can find
the path here in $\lceil(n-1)/l\rceil=\lceil n/l\rceil-1$
rounds. This path together with the first $l$ edges gives the path
we were looking for in the original pyramidal graph and we had
$1+\lceil n/l\rceil-1=\lceil n/l\rceil$ rounds as needed. To
start the recursion we need that if $n\le l$ then one round is
enough. This is trivially true as in one round we can ask all the
vertices that are not sinks, and so we can determine the path.
\end{proof}

The main result of this section is the following. We give a general lower bound that verifies Conjecture \ref{mainconj} for $l=2$ (i.e. $k=s_l=3$) and
solves the case $k=2$.

\begin{thm} \label{pypathm}
For arbitrary $k$, $pa_k(Py(n))\ge si_k(Py(n))\ge
\lceil\frac{2}{k+1}n\rceil$.
\end{thm}

\begin{proof}
We give two different
proofs. While processing an algorithm which finds the path/sink, there is always a maximal partial
pyramid path $P_{max}$ that we know from the answers until now,
i.e. the path determined by the state of the already known
vertices (note that $P_{max}$ is changing by time). The basic idea
in both proofs is that in each round there is only one question
which immediately makes $P_{max}$ longer by one and for the rest
of the questions, only pairs of them can determine one more edge
in $P_{max}$. In both proofs the adversary has the following
answering scheme. In each round he answers for the $k$ asked
vertices in reverse order of their height i.e. he first answers
for the one which is on a level with a biggest index (if there are more asked vertices on
the same level, then their order does no matter), etc.. This way at most
one vertex per round is the endvertex of $P_{max}$ when it is
asked.

{\em First proof of the lower bound.} If a vertex is not an
endvertex of $P_{max}$ then we just answer left. If the asked
vertex $v$ is an endvertex of $P_{max}$ then we do the following.
Let $v$'s left child be $u$ and its right child be $w$. Compute
the length $l_u$ of the path starting from $u$ determined by the
already known states of vertices (when we reach a vertex with
unknown state, that's the end of the path, eg. it may be already
$u$ if we don't know $u$'s state). Similarly, the length of the
path starting at $w$ is $l_w$. Now we answer left for the state of
$v$ if $l_w\le l_u$ and right otherwise i.e. we choose to go in
the direction where the continuation of the path will be shorter.

When analyzing this method we just concentrate on the vertices
which at the stage when they are asked, are the endpoints of the
current $P_{max}$. As already mentioned, we consider only one such
point in each round. After our answer to a vertex $v$ the path
gets longer by at least one. If by more than one, then wlog. we
have chosen $u$, its left child and so $l_u\le l_w$, where $w$ is
its right child. As for non-endpoints of $P_{max}$ we always
answer left and the two paths starting at $u$ and $w$ contain only
such vertices, they are completely disjoint. I.e. for all but one
edge of the new $P_{max}$ determined in this round, we found two
vertices for which the state was asked already. As $P_{max}$ is
increasing, in each round we find new such pairs of vertices. It
is easy to see that doing this the sink is determined if and only
if the pyramid path is determined as well. Suppose now that after
$m$ rounds the whole pyramid path is determined, i.e. all $n$
edges of it. In each round there was at most one endvertex of
$P_{max}$ asked, which means that at most $m$ edges were
determined by them and for the rest of the edges we found two
asked vertices for each. Thus together there where at least
$m+2(n-m)=2n-m$ questions. Thus we had at least $m\ge
(2n-m)/k=(2n-m)/k$ rounds which implies $m\ge
\frac{2}{k+1}n$ as needed.

{\em Second proof of the lower bound.} If a vertex $v$ is not an
endvertex when it is asked then we check if there is another
vertex $v'$ on the same level with known state. If there is, then
we give the same answer for $v$ as we gave for $v'$ and
additionally we tell that the state of every vertex on this level
is the same. Otherwise we give an arbitrary answer and also say
that either there will be one more asked vertex in this level or
we will avoid $v$. If $v$ is an endvertex of $P_{max}$, then we
determine the first level under it for which we did not tell the
state of every vertex on that level. On this level either there is
no asked vertex, then our answer to $v$ is arbitrary or there is a
vertex $q$ which was already asked. Now either answering left or
right to $v$ will make sure that the endvertex of the new
$P_{max}$ is on the same level as $q$ but a different vertex. Thus
the question when we asked $q$ became useless, wlog. we can assume
that there is no such $q$. This way questions which were
endvertices of $P_{max}$ determine one edge in $P_{max}$ and pairs
of the rest of the vertices determine a whole level, i.e. one edge
in $P_{max}$. It is easy again to see that doing this the sink is
determined if and only if the pyramid path is determined as well.
The same computation as in the first proof yields the desired
lower bounds.
\end{proof}

\begin{cor}
$pa_{2}(Py(n))=si_{2}(Py(n))=\lceil\frac{2}{3}n\rceil$.\\
$pa_{3}(Py(n))=si_{3}(Py(n))=\lceil\frac{1}{2}n\rceil$.
\end{cor}
\begin{proof}
Theorem \ref{pypathm} implies $si_{2}(Py(n))\ge \lceil\frac{2}{3}n\rceil$ and $si_{3}(Py(n))\ge\lceil\frac{1}{2}n\rceil$.
We now need to give algorithms for finding the paths, that achieve these bounds.
For $k=3$ the algorithm in Claim \ref{upperbound} can be applied.
For $k=2$ first we ask the root $v_{1,1}$ and $v_{3,2}$. Wlog. the
root is in left state. Now in the second round we ask $v_{2,1}$
and $v_{3,1}$. After these two rounds we will know the first $3$
edges of the path and then we can proceed by recursion (taking the
endvertex of this $3$ long path as the new root).
\end{proof}

Pyramid graphs can be easily generalized to
$d$-dimensions, see eg. the paper of Sun et al. \cite{sun} In this paper a pyramid graph
is represented on the non-negative part of the
$2$-dimensional grid with the origo being its root. In a similar
way a $d$-dimensional pyramid graph is a part of the
$d$-dimensional grid. In the following we give a generalization of pyramid paths that includes the $d$-dimensional pyramid path.
A {\em generalized pyramid graph} $GPy_d(n)$ is a directed graph having the following properties. $GPy_d(n)$ has its vertices on $n+1$ levels such that the first level has one source vertex and the last level contains only sinks. From any vertex $v_i$ on level $i\le n$, there are $d$ outgoing edges to level $i+1$, and between each two levels $i$ and $i+1$ there is a matching $L_i$ that matches level $i$ to level $i+1$. As a consequence, on each non-first level there are at least $d$ vertices.

\begin{thm} \label{gpypathm}
For any generalized pyramid graph $GPy_d(n)$ and arbitrary $k$, $pa_k(Py(n))\ge si_k(Py(n))\ge
\lceil\frac{d}{k-1+d}n\rceil$.
\end{thm}

\begin{proof}
Both proofs of Theorem \ref{pypathm} easily generalize to this setting. Here we present a proof using the second method. We refer to the edges of all the $L_i$'s as left edges.  If a vertex $v$ is not an
endvertex when it is asked then we check if there are at least $d-1$ another
vertices on the same level with known state. If no, then we give answer left (i.e. the edge from the appropriate matching) and also say
that either there will be one more asked vertex in this level or
the final path won't go through $v$. If there are at least $d-1$ vertices already on this level with known state, then
we again give answer left for $v$ and
additionally we tell that the state of every vertex on this level
is left. If $v$ is an endvertex of $P_{max}$ which is on level $i$, then we
determine the first level $j$ under it for which we did not tell the
state of every vertex on that level. If on this level there are
at most $d-1$ asked vertices, then our answer to $v$ is such that the new endvertex of $P_{max}$ is a vertex on level $j$ that was not yet asked. This can be done, as the at least $d$ different choices for the state of $v$ all yield to different endvertices on all the levels from $i$ to $j$, as on every level every known state is a matching edge and the matching edges never go to the same vertex.

It is again easy to conclude that in each round there is only one question
which immediately makes $P_{max}$ longer by one and for the rest
of the questions, only $d$-tuples of them can determine one more edge
in $P_{max}$. Simple computation gives the lower bound $si_k(GPy_d(n))\ge\frac{d}{k-1+d}n$.
\end{proof}

We also remark that Theorem \ref{gpypathm} cannot be improved as there are generalized pyramid graphs for every $d$ where equality holds. Indeed, take the $n+1$ level generalized pyramid graph having $d$ vertices on each non-first level for which every non-sink vertex is connected to every vertex on the next level. This graph is uniquely determined by $n$ and $d$. Now in this graph an obvious algorithm is to ask in each round the endvertex of $P_{max}$ and as many complete levels under it as possible. This way in each round $P_{max}$ gets longer by $1+\lfloor (k-1)/d\rfloor=\lfloor (k-1+d)/d\rfloor$, thus we can determine the path in at most $\lceil n/(\lfloor \frac{k-1+d}{d}\rfloor) \rceil$ rounds, If $k-1$ is divisible by $d$ then this upper bound matches the lower bound of Theorem \ref{gpypathm}.

%

\section{Remarks}

Let us consider again the problem of $pa_k(G)$ and $si_k(G)$ for
directed acyclic graphs. Obviously the same preprocessing as in
the case $k=1$, replacing $G$ by $G'$ or $G''$, is useful in
general. From now on we suppose that $G$ does not have vertices
with outdegree 1 (which is always true for $G'$ and $G''$).

One could ask how far $si(G)$ and $si_k(G)$ (or $pa(G)$ and
$pa_k(G)$) can be. Obviously $si(G)\le k si_k(G)$ as the same at
most $k si_k(G)$ questions which were used to find the sink in the
case of $si_k(G)$ could be used one-by-one to find the sink in the
case of $si(G)$. An example where this bound is achieved is the
graph $H_l$ consisting of a directed path of length $kl$ with each
vertex on the path having another out-neighbor, which is a sink.
More precisely let $x_1, \dots, x_{kl+1}$ be vertices of $H_l$
such that $x_ix_{i+1}\in E$ for every $i \le kl$. Additionally,
every $x_i$ with $i \le kl$ has a child, which is a sink. One can
easily see that in the worst case (that can be forced by the
adversary) the sink is found if and only if $x_1, \dots, x_{kl}$
have been asked, hence $si(H_l)=kl$ and $si_k(H_l)=l$. The same is
true if we want to find the path.

An example where $si(G)$ and $si_k(G)$ (or $pa(G)$ and $pa_k(G)$)
are close to each other is the complete $k$-ary tree from Section
\ref{tree}. It follows easily from Theorem \ref{treethm} that
$si(T_k(n))=si_k(T_k(n))=pa(T_k(n))=pa_k(T_k(n))=n$.

More generally, one could ask how far $si_m(G)$ and
$si_k(G)$ (or $pa_m(G)$ and $pa_k(G)$) can be for any constants $m\le k$. Obviously
$si_m(G)\ge si_k(G)$, and similarly to the arguments used in the
case $m=1$, $si_m(G)\le \lceil k/m \rceil si_k(G)$.

One can easily construct a graph where this bound is achieved. We
just mention the main ideas without details. Let us consider a
$k$-ary tree with $l$ levels and replace each vertex with a copy
of $H_1$. The trivial algorithm is to go through the copy of $H$
corresponding to the source of the tree, then in the appropriate
child of it, and so on. This gives $si_k(G)\le pa_k(H')\le l$ and
$si_m(H')\le \lceil k/m \rceil l$. On the other hand the method of
the adversary can be the following: if a vertex is asked and it is
not in the upper-most copy of $H_1$, the path won't even go into
that copy of $H_1$ which contains this vertex. It shows that there
is equality in the previous inequalities.




Summarizing:
\begin{claim}
For arbitrary $m\le k$ $si_m(G)\le \lceil k/m \rceil si_k(G)$ and there are infinitely many $G$ graphs for which equality holds.
\end{claim}

For larger $k$, one can easily improve the trivial
(and optimal) algorithm we mentioned in Section \ref{dag} for
$k=1$. At first we ask the source, all its out-neighbors, every
vertex which can be reached from the source in a path of length two
and so on. If there is an $i$ for which we cannot ask every vertex
which can be reached from the source in a path of length $i$, then
we ask as many as we can, chosen arbitrarily. Then the answers
show the beginning of the path $P(E^*)$. We repeat this procedure
starting with the last vertex which is surely in $P(E^*)$. This
algorithm (let us call it Algorithm A) finds the sink and the path
too, using at most $l(G)$ questions.

Clearly a smarter algorithm cannot be more than $k$ times faster
than this trivial one for a graph (without vertices of outdegree
1), as it would mean $si_k(G)<l(G)/k=si(G)/k$.

Now for any $k$ we show a graph $G$ and algorithm which can
achieve this bound (depending on how the arbitrary vertices are
chosen in the trivial algorithm). Let $x_1, \dots, x_{kl+1}$ be
vertices of the graph such that $x_ix_{i+1}\in E$ for every $i \le
kl$. Additionally, every $x_i$ with $i<kl$ has $k-1$ children,
each of them having two children, and $x_{kl}$ has $k-1$
additional children. These additional vertices are all distinct,
hence the graph is a tree. Algorithm $A$ asks $x_1$ and $k-1$
of its children in the first turn. It is possible that it does not ask
$x_2$. Suppose the path $P(E^*)$ goes to $x_{kl+1}$. If the
arbitrary vertices are always chosen the worst possible way, than
$kl$ turns are needed (even if they are chosen smarter, at least
$kl/2$ turns are needed).

However, consider the following Algorithm $B$. At the first turn
we ask $x_1, \dots, x_k$. If the path $P(E^*)$ does not go to
$x_{k+1}$, then we need to ask one more questions to finish the
algorithm, otherwise we continue with $x_{1+1}, \dots$ $\dots, x_{2k}$,
and so on. One can easily see that Algorithm $B$ finishes after at
most $l$ turns.

Summarizing (we denote by $si_A(G)$ the number of steps in which Algorithm A finds the sink):
\begin{claim}
$si_k(G)\ge si_A(G)/k$ and there are infinitely many $G$ graphs for which equality holds.
\end{claim}

Also, as we noted before, our proof that $pa(G)=l(G)$, if there is
no vertex with outdegree $1$, can be interpreted even for graphs
containing cycles, yet we could not prove such a claim for $si(G)$
if $G$ contains a cycle. Even the question is not clear in this
case. We could ask for the sink or cycle where the flow ends,
analogously to the acyclic case. On the other hand if we want to know
every edge of the cycle, it is more similar to $pa(G)$. A possible
goal could be to determine the sink or the last vertex
before/after creating a cycle in the flow.

\begin{problem} \label{cyclesink}
Give an efficient algorithm to determine in this sense $si(G)$ if
$G$ contains a cycle.
\end{problem}

 The most interesting open problem is still to
determine $pa_k(Py(n))$ for every $k$ and $n$ or at least to prove
Conjecture \ref{mainconj}. Further, our results suggest that the
following might be true.

\begin{problem}
Is it true that $pa_k(Py(n))=si_k(Py(n))$ for every $k$ and $n$?
\end{problem}

In a paper of Sun et al. \cite{sun} a very similar problem was investigated.
In their version of the problem, in one question we can ask for a
vertex if it is on the path or not. Let us call the minimal number
of questions for this version $pa'_k(G)$ and $si'_k(G)$. For pyramid graphs,
the completely adaptive (i.e. one questions per round) version,
similarly to our problem,  we need  $n$ rounds. However, they do
not regard the version when we can ask more questions per round.
It is trivial that a question of our kind can be emulated by $3$
questions of their kind (asking the vertex and also its $2$
outgoing neighbors), thus $pa_k(Py(n))\le 3pa'_k(Py(n))$ and
$pa(Py(n))\le 3si'_k(Py(n))$. It would be interesting to know
more about these two new functions.

What Sun et al. investigate is that for $pa'$ and $si'$, eg. for pyramid graphs
algorithms using randomization can find the path much faster than
deterministic ones. This might show one major difference between these
two sets of problems, as in the version we regard, randomization
does not seem to help much. One possible intuition behind this
difference is that a left/right answer just gives a relative
information, which might be completely useless to determine our
path, whereas in their case any answer gives some information
about the path, i.e. whether it goes through that vertex or not.

Also, in their paper this was a major tool to give bounds to
various Local Search Problems. It would be interesting to see
whether our version has similar theoretical applications.\\

\noindent {\bf Acknowledgement.} We thank G\'abor Wiener
for communicating the problem to us.

\end{document}